\documentclass[a4paper,10pt]{article}
\usepackage[english]{babel}
\usepackage{amsfonts,amsmath,amsthm}

\newtheorem{theorem}{Theorem}[section] 
 \newtheorem{lemma}[theorem]{Lemma}

\theoremstyle{definition} 
 \newtheorem{remark}[theorem]{Remark}

\newcommand{\suppr}{\hfuzz=100pt \vfuzz=100pt} \suppr

\newcommand{\re}{\textup{Re}}
\newcommand{\im}{\textup{Im}}

\newcommand{\e}{\textup{e}}
\renewcommand{\i}{\textup{i}}

\newcommand{\NN}{{\mathbb{N}}}

\newcommand{\spQ}{{\mathbb{Q}}}
\newcommand{\spH}{{\mathbb{H}}}
\newcommand{\spT}{{\mathbb{T}}}

\newcommand{\nn}{{\boldsymbol n}}

\newcommand{\balpha}{{\boldsymbol\alpha}}

\newcommand{\spQQ}{{\pmb{\spQ}}}

\newcommand{\param}{{\alpha}}
\newcommand{\fracparam}{{\alpha/n}}

\DeclareMathOperator{\ess}{ess}
\DeclareMathOperator{\rank}{rank}

\begin{document}

\title{Complementary results for the spectral analysis of matrices in Galerkin methods with GB-splines}
\author{Fabio Roman\footnote{University of Torino, Dipartimento di Matematica, Via Carlo Alberto 10, 10123 Torino. Tel: +39 011 6702827; Fax: +39 011 6702878; e-mail: fabio.roman@unito.it}}
\date{}

\maketitle

\section{Introduction}

In this short note we collect some results relative to the study of the spectral analysis of matrices in Galerkin methods based on generalized B-splines with high smoothness, depicted in \cite{galerkin_gbsp}. \\
They are the generalization of the ones reported in \cite{GMPSS14} for the polynomial case, which were studied also in the generalized context, but not inserted into \cite{galerkin_gbsp}.

\section{Estimates for the minimal eigenvalues} \label{sec:mineig} 

In this subsection we provide estimates for the minimal eigenvalues of $ M_{n,p}^{\spQ_\mu} $ and $ K_{n,p}^{\spQ_\mu} $. These estimates will be employed to obtain a lower bound for $ | \lambda_{\min} (A_{n,p}^{\spQ_\mu}) | $, where $ \lambda_{\min} (A_{n,p}^{\spQ_\mu}) $ is an eigenvalue of $ A_{n,p}^{\spQ_\mu} $ with minimum modulus. \\
We begin by generalizing \cite[Eq.~(51)]{GMPSS14}. We remember that the result declares how, for every $ p \geq 1, n \geq 2 $, and $ \mathbf{x} = (x_1, \ldots, x_{n+p-2}) \in \mathbb{R}^{n+p-2}, $, it holds
\begin{equation} \label{eq:poly51_orig}
C_p \frac{\Vert \mathbf{x} \Vert^2}{n} 
\leq \left \Vert \sum_{i=1}^{n+p-2} x_i N_{i+1,p} \right \Vert^2_{L_2([0,1])}
\leq \overline{C}_p \frac{\Vert \mathbf{x} \Vert^2}{n},
\end{equation}
where the constants $ C_p, \overline{C}_p > 0 $ do not depend on $n$ and $\mathbf{x}$ (see also \cite[Eq.~(6.3) and Theorem 9.27]{Sch}. It is possible to infer that, for every $ i = 2, \ldots, n+p-1 $, $ x \in \hbox{int}(\hbox{supp}(N_{i,p}(x))) $, there hold
\begin{align}
K_p^{\spQ_\param} & \leq \frac{N_{i,p}^{\spQ_{n\param}}(x)}{N_{i,p}(x)} \leq \overline{K}_p^{\spQ_\param}
\label{eq:gb_b_ratio_nn}, \\
K_p^{\spQ_{[0,\param]}} & \leq \frac{N_{i,p}^{\spQ_\param}(x)}{N_{i,p}(x)} \leq
\overline{K}_p^{\spQ_{[0,\param]}} \label{eq:gb_b_ratio_n}.
\end{align}
where we use the notation $ \param^* = \frac{\param}{\lfloor \frac{\param}{\pi} \rfloor + 1} $ and
\begin{align*}
K_p^{\spH_{[0,\param]}} &:= \inf_{\overline{\param} \in (0,\param]} K_p^{\spH_{\overline{\param}}}
                        \leq \inf_{n \in \NN_0} K_p^{\spH_\fracparam}, \\
K_p^{\spT_{[0,\param]}} &:= \inf_{\overline{\param} \in (0,\param^*]} K_p^{\spT_{\overline{\param}}}
                        \leq \inf_{n \geq \lfloor \frac{\param}{\pi} \rfloor + 1} K_p^{\spT_\fracparam}, \\
    \overline{K}_p^{\spH_{[0,\param]}} 
&:= \sup_{\overline{\param} \in (0,\param]} K_p^{\spH_{\overline{\param}}}
    \geq \sup_{n \in \NN_0} K_p^{\spH_\fracparam}, \\
    \overline{K}_p^{\spT_{[0,\param]}} 
&:= \sup_{\overline{\param} \in (0,\param^*]} K_p^{\spT_{\overline{\param}}}
    \geq \sup_{n \geq \lfloor \frac{\param}{\pi} \rfloor + 1} K_p^{\spT_\fracparam}.
\end{align*}
Indeed, $ N_{i,p}^{\spQ_{n\param}}(x) $ and $ N_{i,p}(x) $ are zero in the same points, and the zeros in their supports (at the boundaries) have the same order; furthermore, $ N_{i,p}^{\spQ_\param}(x) \rightarrow N_{i,p}(x) $ while $ n \rightarrow \infty $.

As a consequence, with norms in $ L_2([0,1]) $
\begin{align}
       \left \Vert \sum_{i=1}^{n+p-2} x_i K_p^{\spQ_\param} N_{i+1,p} \right \Vert^2
& \leq \left \Vert \sum_{i=1}^{n+p-2} x_i N_{i+1,p}^{\spQ_{n\param}} \right \Vert^2
  \leq \left \Vert \sum_{i=1}^{n+p-2} x_i \overline{K}_p^{\spQ_\param} N_{i+1,p}\right\Vert^2, \\
       \left \Vert \sum_{i=1}^{n+p-2} x_i K_p^{\spQ_{[0,\param]}} N_{i+1,p} \right \Vert^2
& \leq \left \Vert \sum_{i=1}^{n+p-2} x_i N_{i+1,p}^{\spQ_\param} \right \Vert^2
  \leq \left \Vert \sum_{i=1}^{n+p-2} x_i \overline{K}_p^{\spQ_{[0,\param]}} N_{i+1,p}\right\Vert^2,
\end{align}
and so, by setting $ C_p^{\spQ_\param} := (K_p^{\spQ_\param})^2 C_p $, $ \overline{C}_p^{\spQ_\param} := (\overline{K}_p^{\spQ_\param})^2 \overline{C}_p $, $ C_p^{\spQ_{[0,\param]}} := (K_p^{\spQ_{[0,\param]}})^2 C_p $, $ \overline{C}_p^{\spQ_{[0,\param]}} := (\overline{K}_p^{\spQ_{[0,\param]}})^2 \overline{C}_p $
\begin{align}
        C_p^{\spQ_\param} \frac{\Vert \mathbf{x} \Vert^2}{n}
& \leq \left \Vert \sum_{i=1}^{n+p-2} x_i N_{i+1,p}^{\spQ_{n\param}} \right \Vert^2
  \leq \overline{C}_p^{\spQ_\param} \frac{\Vert \mathbf{x} \Vert^2}{n}, \label{eq:poly51_nn} \\
        C_p^{\spQ_{[0,\param]}} \frac{\Vert \mathbf{x} \Vert^2}{n}
& \leq \left \Vert \sum_{i=1}^{n+p-2} x_i N_{i+1,p}^{\spQ_\param} \right \Vert^2
  \leq  \overline{C}_p^{\spQ_{[0,\param]}} \frac{\Vert \mathbf{x} \Vert^2}{n}. \label{eq:poly51_n}
\end{align}
%

%
We also recall the Poincar\'e inequality in the one-dimensional setting:

\begin{equation} \label{eq:poly52}
\Vert v \Vert_{L_2([0,1])} \leq \frac{1}{\pi} \Vert v' \Vert_{L_2([0,1])},
\qquad \forall v \in H_0^1([0,1]),
\end{equation}

where $ \frac{1}{\pi} $ is the best constant, see \cite{bottcher07}. \\
We can use (\ref{eq:poly51_nn})-(\ref{eq:poly51_n}) and (\ref{eq:poly52}) to prove the next theorem, by defining $ C_p^{\spQ_\nu} := C_p^{\spQ_\param} $ in the non-nested case, $ C_p^{\spQ_\nu} := C_p^{\spQ_{[0,\param]}} $ in the nested case.

\begin{theorem} \label{th:poly08}
Let $ C_p^{\spQ_\nu} \geq 0 $ be the constant in (\ref{eq:poly51_nn}) or (\ref{eq:poly51_n}), then for all $ p \geq 2 $ and $ n \geq 2 $ the following properties hold.

1. $ \lambda_{\min}(M_{n,p}^{\spQ_\mu}) \geq C_p^{\spQ_\nu} $,

2. $ K_{n,p}^{\spQ_\mu} \geq \frac{\pi^2}{n^2} M_{n,p}^{\spQ_\mu} $ and $ \lambda_{\min}(K_{n,p}^{\spQ_\mu}) \geq \frac{\pi^2 C_p^{\spQ_\nu}}{n^2} $.
\end{theorem}

\begin{proof}
The proof is analogous to the one referring to the polynomial case, see \cite[Theorem 8]{GMPSS14}, by considering that $ M_{n,p}^{\spQ_\mu} $ and $ K_{n,p}^{\spQ_\mu} $ are still symmetric matrices, and having (\ref{eq:poly51_nn})-(\ref{eq:poly52}).
\end{proof}

\begin{remark}
For every $ p \geq 2, n \geq 2 $ and $ j = 1, \ldots, n+p-2 $, let $ \lambda_j(K_{n,p}^{\spQ_\mu}) $ be the $j$-th smallest eigenvalue of $ K_{n,p}^{\spQ_\mu} $, that is, $ \lambda_1(K_{n,p}^{\spQ_\mu}) \leq \ldots \lambda_{n+p-2}(K_{n,p}^{\spQ_\mu}) $. Then, we conjecture that for every $ p \geq 2 $ and for each fixed $ j \geq 1 $,

\begin{equation}
\lim_{n \rightarrow \infty} (n^2 \lambda_j(K_{n,p}^{\spQ_\mu})) = j^2 \pi^2.
\end{equation}

\end{remark}

There is a motivation related to the connection between $ K_{n,p}^{\spQ_\mu} $ and the problem \cite[Eq. (1)]{galerkin_gbsp}, furthermore in the polynomial case it has been verified numerically for some values of $p$ and $j$; see \cite[Remark 3]{GMPSS14} for details.

\begin{theorem}
For all $ p \geq 2 $ and all $ n \geq 2 $, let $ \lambda_{\min}(A_{n,p}^{\spQ_\mu}) $ be an eigenvalue of $A_{n,p}^{\spQ_\mu}$ with minimum modulus. Then,

\begin{equation} \label{eq:poly56}
| \lambda_{\min}(A_{n,p}^{\spQ_\mu}) | \geq \lambda_{\min}(\re A_{n,p}^{\spQ_\mu}) \geq \frac{C_p^{\spQ_\nu}(\pi^2+\gamma)}{n},
\end{equation}
with $ C_p^{\spQ_\nu} \geq 0 $ being the same constant appearing in Theorem \ref{th:poly08}.

\end{theorem}

\begin{proof}
Let us preliminary recall that, having $X^*$ the conjugate transposed of $X$, they hold:
$$ \re X := \frac{X+X^*}{2}, \qquad \im X := \frac{X-X^*}{2\i}. $$
By the expression 
%
$$ A_{n,p}^{\spQ_\mu} = n K_{n,p}^{\spQ_\mu} + \beta H_{n,p}^{\spQ_\mu}
                        + \frac{\gamma}{n} M_{n,p}^{\spQ_\mu}, $$
and recalling that $K_{n,p}^{\spQ_\mu}, M_{n,p}^{\spQ_\mu}$ are symmetric, while $H_{n,p}^{\spQ_\mu}$ is skew-symmetric, we infer that the real part of $A_{n,p}^{\spQ_\mu}$ is given by

$$ \re A_{n,p}^{\spQ_\mu} = n K_{n,p}^{\spQ_\mu} + \frac{\gamma}{n} M_{n,p}^{\spQ_\mu}. $$
Therefore, by the minimax principle and by Theorem \ref{th:poly08} we obtain

$$ \lambda_{\min}(\re A_{n,p}^{\spQ_\mu}) 
   \geq \lambda_{\min}(n K_{n,p}^{\spQ_\mu}) 
        + \lambda_{\min} \left( \frac{\gamma}{n} M_{n,p}^{\spQ_\mu} \right)
   \geq n \frac{\pi^2 C_p^{\spQ_\nu}}{n^2} + \frac{\gamma}{n} C_p^{\spQ_\nu} 
   = \frac{C_p^{\spQ_\nu}(\pi^2+\gamma)}{n}. $$
The result is obtained by considering that $ | \lambda_{\min}(A_{n,p}^{\spQ_\mu}) | \geq \lambda_{\min}(\re A_{n,p}^{\spQ_\mu}) $ because of the spectrum localization

\begin{equation}
\sigma(X) \subseteq [ \lambda_{\min}(\re X), \lambda_{\max}(\re X) ] 
          \times    [ \lambda_{\min}(\im X), \lambda_{\max}(\im X) ] \subset \mathbb{C}, \qquad
          \forall   X \in \mathbb{C}^{m \times m}.
\end{equation}

\end{proof}

The lower bound (\ref{eq:poly56}) remains bounded away from 0 for all $ \gamma \geq 0 $ and, in particular, for the interesting value $ \gamma = 0 $.

\section{Conditioning} \label{sec:conditioning} 

In this subsection we provide a bound for the condition number 

$$    \kappa_2(A_{n,p}^{\spQ_\mu}) 
   := \Vert A_{n,p}^{\spQ_\mu} \Vert \mbox{ } \Vert (A_{n,p}^{\spQ_\mu})^{-1} \Vert. $$
Let us start with the Fan-Hoffman theorem (\cite{bhatia97}).

\begin{theorem}

Let $ X \in \mathbb{C}^{m \times m} $ and let:
$$ \Vert X \Vert = s_1(X) \geq s_2(X) \geq \cdots \geq s_m(X), \quad
   \lambda_1(\re X) \geq \lambda_2(\re X) \geq \cdots \lambda_m(\re X) $$
be the singular values of $X$ and the eigenvalues of $ \re X $, respectively. Then

$$ s_j(X) \geq \lambda_j(\re X), \quad \forall j = 1, \ldots, m. $$

\begin{theorem}

For every $ p \geq 2 $ there exists a constant $ \alpha_p > 0 $ such that

\begin{equation}
\kappa_2(A_{n,p}^{\spQ_\mu}) \leq \alpha_p n^2, \quad \forall n \geq 2.
\end{equation}

\end{theorem}

\begin{proof}

Fix $ p \geq 2 $ and $ n \geq 2 $. Being either symmetric or skew-symmetric, $ K_{n,p}^{\spQ_\mu}, H_{n,p}^{\spQ_\mu} $ and $ M_{n,p}^{\spQ_\mu} $ are normal matrices, and by applying \cite[Lemma 7]{galerkin_gbsp} we obtain for $ \Vert A_{n,p}^{\spQ_\mu} \Vert $ the following bound:

\begin{eqnarray}
         \Vert A_{n,p}^{\spQ_\mu} \Vert 
& =    & \left\Vert n K_{n,p}^{\spQ_\mu} + \beta H_{n,p}^{\spQ_\mu}
         + \frac{\gamma}{n} M_{n,p}^{\spQ_\mu} \right\Vert
  \leq   \Vert n K_{n,p}^{\spQ_\mu} \Vert + |\beta| \Vert H_{n,p}^{\spQ_\mu} \Vert 
         + \gamma \left\Vert \frac{1}{n} M_{n,p}^{\spQ_\mu} \right\Vert \nonumber \\
& \leq & \Vert n K_{n,p}^{\spQ_\mu} \Vert_{\infty} + |\beta| \Vert H_{n,p}^{\spQ_\mu} \Vert_{\infty}
         + \gamma \left\Vert \frac{1}{n} M_{n,p}^{\spQ_\mu} \right\Vert_{\infty}
  \leq   C_{p,\alpha} n + 2|\beta| + \frac{\gamma(p+1)}{n}.
\end{eqnarray}

\noindent On the other hand, being $ s_{n+p-2}(A_{n,p}^{\spQ_\mu}) $ the minimum singular value of $A_{n,p}^{\spQ_\mu}$, we have

\begin{equation}
   \Vert (A_{n,p}^{\spQ_\mu})^{-1} \Vert 
 = \frac{1}{s_{n+p-2}(A_{n,p}^{\spQ_\mu})} \leq \frac{1}{\lambda_{\min}(\re A_{n,p}^{\spQ_\mu})}
   \leq \frac{n}{C_p^{\spQ_\nu}(\pi^2+\gamma)}.
\end{equation}

\noindent so $$ \displaystyle \kappa_2(A_{n,p}^{\spQ_\mu}) 
\leq \frac{C_{p,\alpha}n^2 + 2n|\beta| + \gamma(p+1)}{C_p^{\spQ_\nu}(\pi^2+\gamma)}
\leq \frac{C_{p,\alpha}n^2 + n^2 |\beta| + \gamma(p+1)(n^2/4)}{C_p^{\spQ_\nu}(\pi^2+\gamma)}. $$
that makes the theorem satisfied for $ \alpha_p := \frac{1}{C_p^{\spQ_\nu}(\pi^2+\gamma)} [C_{p,\alpha} + |\beta| + \frac{\gamma(p+1)}{4}]. $

\end{proof}

\end{theorem}

\section{More preliminaries on spectral analysis}

%
%
%

The following one is another important result regarding sequences of Toeplitz matrices.

\begin{theorem} \label{th:poly03-arrows}

Let $ f \in L_1([-\pi,\pi]) $ be a real-valued function, and let $ m_f := \ess \inf f, M_f := \ess \sup f $, and suppose $ m_f < M_f $. Then

\begin{center}
$ \lambda_{\min}(T_m(f)) \searrow m_f $ and $ \lambda_{\max}(T_m(f)) \nearrow M_f $ as $ m \rightarrow \infty $.
\end{center}

\end{theorem}

Another result due to Parter \cite{parter61} concerns the asymptotics of the $j$-th smallest eigenvalue $ \lambda_j(T_m(f)) $, for $j$ fixed and $ m \rightarrow \infty $.

\begin{theorem}
Let $ f : \mathbb{R} \rightarrow \mathbb{R} $ be continuous and $2\pi$-periodic. Let $ m_f := \min_{\theta \in \mathbb{R}} f(\theta) = f(\theta_{\min}) $ and let $ \theta_{\min} $ be the unique point in $ (-\pi,\pi] $ such that $ f(\theta_{\min}) = m_f $. Assume there exists $ s \geq 1 $ such that $f$ has $2s$ continuous derivatives in $ (\theta_{\min} - \epsilon, \theta_{\min} + \epsilon) $ for some $ \epsilon > 0 $ and $ f^{(2s)}(\theta_{\min}) > 0 $ is the first non-vanishing derivative of $f$ at $\theta_{\min}$. Finally, for every $ m \geq 1 $, let $ \lambda_1(T_m(f)) \leq \cdots \leq \lambda_m(T_m(f)) $ be the eigenvalues of $T_m(f)$ arranged in non-decreasing order. Then, for each fixed $ j \geq 1 $,

$$ \lambda_j(T_m(f)) - m_f \sim_{m \rightarrow \infty} 
   c_{s,j} \frac{f^{(2s)}(\theta_{\min})}{(2s)!} \frac{1}{m^{2s}}, $$
where $ c_{s,j} > 0 $ is a constant depending only on $s$ and $j$.
\end{theorem}

\begin{remark}
The constant $ c_{s,j} $ is the $j$-th smallest eigenvalue of the boundary value problem

$$ \left\{ \begin{aligned}
& (-1)^s u^{(2s)}(x) = \hbox{f}(x), \quad \mbox{for } 0 < x < 1, \\
& u(0) = u'(0) = \cdots = u^{(s-1)}(0) = 0, u(1) = u'(1) = \cdots = u^{(s-1)}(1) = 0,
\end{aligned} \right. $$
see \cite[p. 191]{parter61}. Thus, we find that $ c_{1,j} = j^2 \pi^2 $ for all $ j \geq 1 $.

\end{remark}

\noindent It is also important to recall some properties of Toeplitz matrices having a two-level structure.
In the next lemma, see e.g. \cite{bhatia97}, we collect some basic results concerning tensor-products matrices.
%
%
\begin{lemma} \label{lem:poly01}
Suppose that $ X \in \mathbb{C}^{m_1 \times m_1} $ and $ Y \in \mathbb{C}^{m_2 \times m_2} $ are normal matrices with eigenvalues given by $ \lambda_1(X), \ldots, \lambda_{m_1}(X) $ and $ \lambda_1(Y), \ldots, \lambda_{m_2}(Y) $. Then
\begin{enumerate}
\item $ X \otimes Y $ is normal and $ (X \otimes Y)^* = X^* \otimes Y^* $;
\item $ \sigma(X \otimes Y) = \{ \lambda_i(X) \lambda_j(Y) : i = 1, \ldots, m_1, j = 1, \ldots, m_2 \} $;
\item $ \rank(X \otimes Y) = \rank(X) \rank(Y); $
\item $ \Vert X \otimes Y \Vert = \Vert X \Vert \mbox{ } \Vert Y \Vert $
      and $ \Vert X \otimes Y \Vert_1 = \Vert X \Vert_1 \mbox{ } \Vert Y \Vert_1 $.
\end{enumerate}
Note that statements 1 and 2 imply that if $ X,Y $ are Hermitian then $ X \otimes Y $ is Hermitian, and if $ X,Y $ are Hermitian and positive definite then $ X \otimes Y $ is Hermitian and positive definite.
\end{lemma}


Given a bivariate function $ g : [-\pi,\pi]^2 \rightarrow \mathbb{R} $ belonging to the space $ L_1([-\pi,\pi]^2) $, we can associate to $g$ a family of two-level Hermitian Toeplitz matrices $ \{ T_{m_1,m_2}(g) \} $ parametrized by two indices $m_1,m_2$ and defined blockwise for every $ m_1, m_2 \geq 1 $: 

$$ T_{m_1,m_2}(g) := \begin{bmatrix}
G_0       & G_{-1} & \cdots & \cdots & G_{-(m_1-1)} \\
G_1       & \ddots & \ddots &        & \vdots       \\
\vdots    & \ddots & \ddots & \ddots & \vdots       \\
\vdots    &        & \ddots & \ddots & G_{-1}       \\
G_{m_1-1} & \cdots & \cdots & G_1    & G_0
\end{bmatrix} \in \mathbb{C}^{m_1 m_2 \times m_1 m_2}, $$
where for every $ k \in \mathbb{Z} $,

$$ G_k := \begin{bmatrix}
g_{k,0}     & g_{k,-1} & \cdots & \cdots  & g_{k,-(m_2-1)} \\
g_{k,1}     & \ddots   & \ddots &         & \vdots         \\
\vdots      & \ddots   & \ddots & \ddots  & \vdots         \\
\vdots      &          & \ddots & \ddots  & g_{k,-1}       \\
g_{k,m_2-1} & \cdots   & \cdots & g_{k,1} & g_{k,0}
\end{bmatrix} \in \mathbb{C}^{m_2 \times m_2}, $$
and for every $ k, l \in \mathbb{Z} $,

$$ g_{k,l} := \frac{1}{(2\pi)^2} \int_{-\pi}^{\pi} \int_{-\pi}^{\pi} 
              g(\theta_1,\theta_2) \e^{-\i(k\theta_1+l\theta_2)} d\theta_1 d\theta_2 $$
is the $(k,l)$ Fourier coefficient of $g$. A Szeg\"o-like result, analogous to \cite[Theorem 2]{galerkin_gbsp}, holds for two-level sequences (see \cite{serra94} and \cite{tilli98}).

\begin{theorem} \label{th:poly05}

Let $ g \in L_1([-\pi,\pi]^2) $ be a real-valued function, and let $ m_g := \ess \inf g, M_g := \ess \sup g $, and suppose $ m_g < M_g $. Then:

\begin{itemize}
\item $ \sigma(T_{m_1,m_2}(g)) \subset (m_g,M_g), \quad \forall m_1, m_2 \geq 1 $;
\item $ \{ T_{m_1,m_2}(g) \} \sim_{\lambda} g $.
\end{itemize}

\end{theorem}

\noindent With the last result, we can relate tensor-products and Toeplitz matrices. Given two functions $ f, h : [-\pi,\pi] \rightarrow \mathbb{R} $ in $ L_1([-\pi,\pi]) $, we can construct the tensor-product function, belonging to $ L_1([-\pi,\pi]^2) $:

$$  f \otimes h : [-\pi,\pi]^2 \rightarrow \mathbb{R}, \quad 
   (f \otimes h)(\theta_1,\theta_2) := f(\theta_1) h(\theta_2). $$
So, we can consider the three families of Hermitian Toeplitz matrices $ \{ T_{m_1}(f) \}, \{ T_{m_2}(h) \} $ and $ \{ T_{m_1,m_2}(f \otimes h) \} $. By computing directly we obtain a commutative property between the operation of tensor-product and the Toeplitz operator.

\begin{lemma} \label{lem:poly02}

Let $ f, h \in L_1([-\pi,\pi]) $ be real-valued functions. Then, for all $ m_1, m_2 \geq 1 $,

$$ T_{m_1}(f) \otimes T_{m_2}(h) = T_{m_1,m_2}(f \otimes h). $$

\end{lemma}

\section{More results on spectral distribution}

The following results are a consequence of Theorem \ref{th:poly03-arrows}.

\begin{lemma} \label{specdist_more}
They hold:
\begin{enumerate}
\item $ \lambda_{\min}(B_{n,p}^{\spQ_{\alpha}}) \searrow 0 $ and $ \lambda_{\max}(B_{n,p}^{\spQ_{\alpha}}) \rightarrow M_{f_p} $ as $ n \rightarrow \infty $; 
\item for each fixed $ j \geq 1 $,
$$ \lambda_j(B_{n,p}^{\spQ_{\alpha}}) \sim_{n \rightarrow \infty} \frac{j^2 \pi^2}{n^2}, $$
where $ \lambda_1(B_{n,p}^{\spQ_{\alpha}}) \leq \cdots \leq \lambda_{n+p-2}(B_{n,p}^{\spQ_{\alpha}}) $ are the eigenvalues of $B_{n,p}^{\spQ_{\alpha}}$ in non-decreasing order; 
\item $ \lambda_{\min}(B_{n,p}^{\spQ_{n\alpha}}) \searrow 0 $ and $ \lambda_{\max}(B_{n,p}^{\spQ_{n\alpha}}) \nearrow M_{f_p^{\spQ_{\alpha}}} $ as $ n \rightarrow \infty $; 
\item $ \lambda_{\min}(C_{n,p}^{\spQ_{\alpha}}) \rightarrow m_{h_p} $ and $ \lambda_{\max}(C_{n,p}^{\spQ_{\alpha}}) \nearrow 1 $ as $ n \rightarrow \infty $; 
\item $ \lambda_{\min}(C_{n,p}^{\spQ_{n\alpha}}) \searrow m_{h_p^{\spQ_{\alpha}}} $ and $ \lambda_{\max}(C_{n,p}^{\spQ_{n\alpha}}) \nearrow 1 $ as $ n \rightarrow \infty $. 
\end{enumerate}
\end{lemma}


\section{Two-dimensional linear elliptic differential problem} 

We can extend the problem \cite[Eq. (1)]{galerkin_gbsp} to the \texttt{2D} domain $ \Omega = (0,1)^2 $, that is:

\begin{equation} \left\{ \begin{aligned} \label{eq:poly83}
& -\Delta u(x,y) + \beta \cdot \nabla u(x,y) + \gamma u(x,y) = f(x,y), \qquad \forall(x,y) \in \Omega, \\
& u(x,y) = 0, \hspace{6.72cm} \forall(x,y) \in \partial \Omega.
\end{aligned} \right. 
\end{equation}
with $ f \in L^2((0,1)^2), \mbox{ } \beta = [\beta_1 \mbox{ } \beta_2]^T \in \mathbb{R}^2,
                           \mbox{ } \gamma \geq 0 $.
We want to approximate the weak solution of (\ref{eq:poly83}) by using the Galerkin method \cite[Eq. (3)]{galerkin_gbsp}; $\mathcal{W}$, approximation space, is now chosen in a manner that takes account of the two-dimensional structure. \\
For the sake of simplicity, we suppose to work in the nested case; we briefly address also the non-nested case while considering, at the end, the main result of this Section, which will be proposed for both cases.
\\ Also, for a better comparison with this study in the polynomial case, we use the inverse lexicographical ordering of \cite[Section 5]{GMPSS14}, rather than the standard lexicographical ordering used in \cite[Section 6]{galerkin_gbsp}. \\
In light of this, we consider two univariate B-spline bases from \cite[Section 4]{galerkin_gbsp}, one for every direction ($x$ and $y$): the basis $ \{ N_{i,p_1}^{\spQ_{\alpha_1}}(x), i = 1, \ldots, n_1+p_1 \} $ over the knot sequence
\begin{equation*}
s_1 = \cdots = s_{p_1+1} = 0 < s_{p_1+2} < \cdots < s_{p_1+n_1} < 1 = s_{p_1+n_1+1} = \ldots = s_{2p_1+n_1+1},
\end{equation*}
where
\begin{equation*}
s_{p_1+i+1} := \frac{i}{n_1}, \qquad \forall i = 0, \ldots, n_1,
\end{equation*}
and the basis $ \{ N_{i,p_2}^{\spQ_{\alpha_2}}(y), i = 1, \ldots, n_2+p_2 \} $ over the knot sequence
\begin{equation*}
t_1 = \cdots = t_{p_2+1} = 0 < t_{p_2+2} < \cdots < t_{p_2+n_2} < 1 = t_{p_2+n_2+1} = \ldots = t_{2p_2+n_2+1},
\end{equation*}
where
\begin{equation*}
t_{p_2+i+1} := \frac{i}{n_2}, \qquad \forall i = 0, \ldots, n_2.
\end{equation*} 
With these, we construct our bivariate tensor-product B-spline basis $ \{ N_{(i,j),(p_1,p_2)}^{\spQQ_{\balpha}}, i = 1, \ldots, n_1+p_1, j = 1, \ldots, n_2+p_2 \} $, like
\begin{equation*}
   N_{(i,j),(p_1,p_2)}^{\spQQ_{\balpha}}(x,y) 
:= ( N_{i,p_1}^{\spQ_{\alpha_1}} \otimes N_{j,p_2}^{\spQ_{\alpha_2}} )(x,y)
 = N_{i,p_1}^{\spQ_{\alpha_1}}(x) N_{j,p_2}^{\spQ_{\alpha_2}}(y).
\end{equation*}
The space $\mathcal{W}$ in the Galerkin problem \cite[Eq. (3)]{galerkin_gbsp} is now chosen as $ \mathcal{W}_{(n_1,n_2),(p_1,p_2)}^{\spQQ_{\balpha}} $, where
\begin{equation} \label{eq:poly84}
   \mathcal{W}_{(n_1,n_2),(p_1,p_2)}^{\spQQ_{\balpha}} 
:= \langle N_{(i,j),(p_1,p_2)}^{\spQQ_{\balpha}} : i = 2, \ldots, n_1+p_1-1, \mbox{ } j = 2, \ldots, n_2+p_2-1 \rangle.
\end{equation}
We order lexicographically the element of the basis (\ref{eq:poly84}), as follows
\begin{equation} \label{eq:poly85}
\varphi_{(n_1+p_1-2)(j-1)+i}^{\spQQ_{\balpha}} := N_{(i+1,j+1),(p_1,p_2)}^{\spQQ_{\balpha}}
\end{equation}
with $ i = 1, \ldots, n_1+p_1-2, \mbox{ } j = 1, \ldots, n_2+p_2-2 $. \\
Once we have fixed the basis (\ref{eq:poly84}) and the order (\ref{eq:poly85}), we devise a linear system \cite[Eq. (4)]{galerkin_gbsp} from the Galerkin problem \cite[Eq. (3)]{galerkin_gbsp}. The stiffness matrix $A$ in \cite[Eq. (4)]{galerkin_gbsp} depends from $n_1,n_2$ and $p_1,p_2$, and can be therefore denoted by:
\begin{equation*}
  A_{(n_1,n_2),(p_1,p_2)}^{\spQQ_{\balpha}} := A_{\bf n,p}^{\spQQ_{\balpha}} := A 
= [ a(\varphi_j^{\spQQ_{\balpha}},\varphi_i^{\spQQ_{\balpha}}) ]_{i,j=1}^{(n_1+p_1-2)(n_2+p_2-2)},
\end{equation*}
where $ a(u,v) =   \int_0^1 \int_0^1 \nabla u \cdot \nabla v \mbox{ } dx dy 
                 + \beta \int_0^1 \int_0^1 \nabla u \mbox{ } v \mbox{ } dx dy
                 + \gamma \int_0^1 \int_0^1 u v \mbox{ } dx dy $,
see \cite[Section 2.1.1]{galerkin_gbsp}

\subsection{Construction of the matrices $ A_{(n_1,n_2),(p_1,p_2)}^{\spQ_{\balpha}} $}
\label{sec:annpp} 

Using the same integration rules of the one-dimensional case, we obtain
\begin{equation} \label{eq:poly86}
A_{\bf n,p}^{\spQQ_{\balpha}} = \frac{n_1}{n_2} \widehat{K}_{\bf n,p}^{\spQQ_{\balpha}}
                               + \frac{n_2}{n_1} \widetilde{K}_{\bf n,p}^{\spQQ_{\balpha}}
                               + \frac{\beta_1}{n_2} \widehat{H}_{\bf n,p}^{\spQQ_{\balpha}}
                               + \frac{\beta_2}{n_1} \widetilde{H}_{\bf n,p}^{\spQQ_{\balpha}}
                               + \frac{\gamma}{n_1 n_2} M_{\bf n,p}^{\spQQ_{\balpha}},
\end{equation}
where
\begin{multline*}
   \widehat{K}_{\bf n,p}^{\spQQ_{\balpha}} 
:= M_{n_2,p_2}^{\spQ_{\alpha_2}} \otimes K_{n_1,p_1}^{\spQ_{\alpha_1}}, \quad
   \widetilde{K}_{\bf n,p}^{\spQQ_{\balpha}} 
:= K_{n_2,p_2}^{\spQ_{\alpha_2}} \otimes M_{n_1,p_1}^{\spQ_{\alpha_1}}, \\
   \widehat{H}_{\bf n,p}^{\spQQ_{\balpha}} 
:= M_{n_2,p_2}^{\spQ_{\alpha_2}} \otimes H_{n_1,p_1}^{\spQ_{\alpha_1}}, \quad
   \widetilde{H}_{\bf n,p}^{\spQQ_{\balpha}} 
:= H_{n_2,p_2}^{\spQ_{\alpha_2}} \otimes M_{n_1,p_1}^{\spQ_{\alpha_1}},
\end{multline*}
$  M_{\bf n,p}^{\spQQ_{\balpha}} 
:= M_{n_2,p_2}^{\spQ_{\alpha_2}} \otimes M_{n_1,p_1}^{\spQ_{\alpha_1}} $. \\ \\
A simplifying case is given by $ n_1 = n_2 = n $ and $ p_1 = p_2 = p $, for which, by defining $ K_{\bf n,p}^{\spQQ_{\balpha}} := \widehat{K}_{\bf n,p}^{\spQQ_{\balpha}} + \widetilde{K}_{\bf n,p}^{\spQQ_{\balpha}} $
\begin{equation}
  A_{\bf n,p}^{\spQQ_{\balpha}} 
= K_{\bf n,p}^{\spQQ_{\balpha}} + \frac{\beta_1}{n} \widehat{H}_{\bf n,p}^{\spQQ_{\balpha}}
  + \frac{\beta_2}{n} \widetilde{H}_{\bf n,p}^{\spQQ_{\balpha}} 
  + \frac{\gamma}{n^2} M_{\bf n,p}^{\spQQ_{\balpha}} 
\end{equation}

\subsection{Spectral distribution} \label{sec:spectral_2d}

We can study now, by fixing $ p_1, p_2 \geq 1 $, the spectral distribution of the sequence of the matrices (\ref{eq:poly86}). We mildly assume that the ratio $ \frac{n_2}{n_1} =: \nu $ is constant as $ n_1 \rightarrow \infty $; in this way, $ A_{\bf n,p}^{\spQQ_{\balpha}} $ is an actual sequence of matrices, since only $n_1$ is a free parameter (we can set for $n_2$ the integer value for which the ratio $\nu$ is better approximated). \\
This condition could be released in favor of some lighter requests, but for the sake of simplicity (especially in notation) we choose to work in this framework, which is ultimately not particulary restrictive.

\begin{equation} \label{eq:poly88}
A_{\bf n,p}^{\spQQ_{\balpha}} = \frac{1}{\nu} \widehat{K}_{\bf n,p}^{\spQQ_{\balpha}}
                               + \nu \widetilde{K}_{\bf n,p}^{\spQQ_{\balpha}}
                               + \frac{\beta_1}{\nu n_1} \widehat{H}_{\bf n,p}^{\spQQ_{\balpha}}
                               + \frac{\beta_2}{n_1} \widetilde{H}_{\bf n,p}^{\spQQ_{\balpha}}
                               + \frac{\gamma}{\nu (n_1)^2} M_{\bf n,p}^{\spQQ_{\balpha}}.
\end{equation}
For every $ n_1 \geq 3p_1 + 1 $ such that $ n_2 = \nu n_1 \geq 3 p_2 + 1 $, we decompose the matrices $ \widehat{K}_{\bf n,p}^{\spQQ_{\balpha}} $ and $ \widetilde{K}_{\bf n,p}^{\spQQ_{\balpha}} $ into
\begin{equation} \label{eq:poly89}
  \widehat{K}_{\bf n,p}^{\spQQ_{\balpha}} 
= \widehat{B}_{\bf n,p}^{\spQQ_{\balpha}} + \widehat{R}_{\bf n,p}^{\spQQ_{\balpha}}, \quad
  \widetilde{K}_{\bf n,p}^{\spQQ_{\balpha}} 
= \widetilde{B}_{\bf n,p}^{\spQQ_{\balpha}} + \widetilde{R}_{\bf n,p}^{\spQQ_{\balpha}},
\end{equation}
where
\begin{equation*}
   \widehat{B}_{\bf n,p}^{\spQQ_{\balpha}}   
:= C_{n_2,p_2}^{\spQ_{\alpha_2}} \otimes B_{n_1,p_1}^{\spQ_{\alpha_1}}, \quad
   \widetilde{B}_{\bf n,p}^{\spQQ_{\balpha}} 
:= B_{n_2,p_2}^{\spQ_{\alpha_2}} \otimes C_{n_1,p_1}^{\spQ_{\alpha_1}},
\end{equation*}
and
\begin{align*}
     \widehat{R}_{\bf n,p}^{\spQQ_{\balpha}} 
& := \widehat{K}_{\bf n,p}^{\spQQ_{\balpha}} - \widehat{B}_{\bf n,p}^{\spQQ_{\balpha}}
   =   C_{n_2,p_2}^{\spQ_{\alpha_2}} \otimes R_{n_1,p_1}^{\spQ_{\alpha_1}} 
     + S_{n_2,p_2}^{\spQ_{\alpha_2}} \otimes B_{n_1,p_1}^{\spQ_{\alpha_1}}
     + S_{n_2,p_2}^{\spQ_{\alpha_2}} \otimes R_{n_1,p_1}^{\spQ_{\alpha_1}}, \\
     \widetilde{R}_{\bf n,p}^{\spQQ_{\balpha}} 
& := \widetilde{K}_{\bf n,p}^{\spQQ_{\balpha}} - \widetilde{B}_{\bf n,p}^{\spQQ_{\balpha}}             
   =   B_{n_2,p_2}^{\spQ_{\alpha_2}} \otimes S_{n_1,p_1}^{\spQ_{\alpha_1}}
     + R_{n_2,p_2}^{\spQ_{\alpha_2}} \otimes C_{n_1,p_1}^{\spQ_{\alpha_1}}
     + R_{n_2,p_2}^{\spQ_{\alpha_2}} \otimes S_{n_1,p_1}^{\spQ_{\alpha_1}},
\end{align*}
where the matrices $ B_{n_i,p_i}^{\spQ_{\alpha_i}}, R_{n_i,p_i}^{\spQ_{\alpha_i}}, C_{n_i,p_i}^{\spQ_{\alpha_i}}, S_{n_i,p_i}^{\spQ_{\alpha_i}} $, $ i = 1,2 $, are those of \cite[Eq. (67)]{galerkin_gbsp} and \cite[Eq. (75)]{galerkin_gbsp}. We further define:
\begin{align}
B_{\bf n,p}^{\spQQ_{\balpha}} & := \frac{1}{\nu} \widehat{B}_{\bf n,p}^{\spQQ_{\balpha}} 
                                  + \nu \widetilde{B}_{\bf n,p}^{\spQQ_{\balpha}}, \\
R_{\bf n,p}^{\spQQ_{\balpha}} & := \frac{1}{\nu} \widehat{R}_{\bf n,p}^{\spQQ_{\balpha}} 
                                  + \nu \widetilde{R}_{\bf n,p}^{\spQQ_{\balpha}}.
\end{align} 
Lemmas \cite[Lemma 6]{galerkin_gbsp} and \cite[Lemma 10]{galerkin_gbsp} state that $ B_{n_i,p_i}^{\spQ_{\alpha_i}} = T_{n_i+p_i-2}(f_{p_i}^{\spQ_{\alpha_i/n_i}}) $ and $ C_{n_i,p_i}^{\spQ_{\alpha_i}} = T_{n_i+p_i-2}(h_{p_i}^{\spQ_{\alpha_i/n_i}}) $, for $ i = 1,2 $, $ p_i \geq 1 $ and $ n_i \geq 3p_i+1 $, thanks to which by Lemma \ref{lem:poly02} we obtain
\begin{align*}
    \widehat{B}_{\bf n,p}^{\spQQ_{\balpha}}   
& = T_{n_2+p_2-2}(h_{p_2}^{\spQ_{\alpha_2/n_2}}) \otimes T_{n_1+p_1-2}(f_{p_1}^{\spQ_{\alpha_1/n_1}})
  = T_{n_2+p_2-2,n_1+p_1-2}(h_{p_2}^{\spQ_{\alpha_2/n_2}} \otimes f_{p_1}^{\spQ_{\alpha_1/n_1}}), \\
    \widetilde{B}_{\bf n,p}^{\spQQ_{\balpha}} 
& = T_{n_2+p_2-2}(f_{p_2}^{\spQ_{\alpha_2/n_2}}) \otimes T_{n_1+p_1-2}(h_{p_1}^{\spQ_{\alpha_1/n_1}})
  = T_{n_2+p_2-2,n_1+p_1-2}(f_{p_2}^{\spQ_{\alpha_2/n_2}} \otimes h_{p_1}^{\spQ_{\alpha_1/n_1}}).
\end{align*}
and, by taking account of the linearity of the operator $ T_{n_2+p_2-2,n_1+p_1-2} $:
\begin{equation}
    B_{\bf n,p}^{\spQQ_{\balpha}} 
=   T_{n_2+p_2-2,n_1+p_1-2} \left( \frac{1}{\nu} h_{p_2}^{\spQ_{\alpha_2/n_2}} \otimes f_{p_1}^{\spQ_{\alpha_1/n_1}}
  + \nu f_{p_2}^{\spQ_{\alpha_2/n_2}} \otimes h_{p_1}^{\spQ_{\alpha_1/n_1}} \right).
\end{equation}
We can use now Theorem \ref{th:poly05} to state
\begin{equation*}
\{ \widehat{B}_{\bf n,p}^{\spQQ_{\balpha}} \} \sim_{\lambda} h_{p_2} \otimes f_{p_1}, \quad
\{ \widetilde{B}_{\bf n,p}^{\spQQ_{\balpha}} \} \sim_{\lambda} f_{p_2} \otimes h_{p_1},
\end{equation*}
and
\begin{equation} \label{eq:poly93}
\{ B_{\bf n,p}^{\spQQ_{\balpha}} \} \sim_{\lambda}
\frac{1}{\nu} h_{p_2} \otimes f_{p_1} + \nu f_{p_2} \otimes h_{p_1}.
\end{equation}
By Lemma \ref{lem:poly01} and the inequalities \cite[Eq. (68)]{galerkin_gbsp} and \cite[Eq. (76)]{galerkin_gbsp}, we have
{\small
\begin{align*}
         \rank(\widehat{R}_{\bf n,p}^{\spQQ_{\balpha}})
& \leq   \rank(C_{n_2,p_2}^{\spQ_{\alpha_2}} \otimes R_{n_1,p_1}^{\spQ_{\alpha_1}}) 
       + \rank(S_{n_2,p_2}^{\spQ_{\alpha_2}} \otimes B_{n_1,p_1}^{\spQ_{\alpha_1}})
       + \rank(S_{n_2,p_2}^{\spQ_{\alpha_2}} \otimes R_{n_1,p_1}^{\spQ_{\alpha_1}}) \\
& =      \rank(C_{n_2,p_2}^{\spQ_{\alpha_2}}) \rank(R_{n_1,p_1}^{\spQ_{\alpha_1}}) 
       + \rank(S_{n_2,p_2}^{\spQ_{\alpha_2}}) \rank(B_{n_1,p_1}^{\spQ_{\alpha_1}})
       + \rank(S_{n_2,p_2}^{\spQ_{\alpha_2}}) \rank(R_{n_1,p_1}^{\spQ_{\alpha_1}}) \\
& \leq   \rank(\nu n_1 + p_2 - 2) \cdot (4p_1-2) + (4p_2-2) \cdot (n_1+p_1-2) + (4p_2-2) \cdot (4p_1-2) \\
& =      o((n_1 + p_1 - 2)(\nu n_1 + p_2 - 2)), \quad \mbox{ as } n_1 \rightarrow \infty.
\end{align*} }
and an analogous result holds for $ \widetilde{R}_{\bf n,p}^{\spQQ_{\balpha}} $ , that is, as $ n_1 \rightarrow \infty $
$$ \rank( \widetilde{R}_{\bf n,p}^{\spQQ_{\balpha}}) = o((n_1 + p_1 - 2)(\nu n_1 + p_2 - 2)). $$
Combining the two previous results gives, again for $ n_1 \rightarrow \infty $
\begin{align}
\rank(R_{\bf n,p}^{\spQQ_{\balpha}}) 
& \leq \rank(\widehat{R}_{\bf n,p}^{\spQQ_{\balpha}}) + \rank(\widetilde{R}_{\bf n,p}^{\spQQ_{\balpha}}) \nonumber \\
& =    o((n_1 + p_1 - 2)(\nu n_1 + p_2 - 2)).
\end{align}
Note that $ (n_1 + p_1 - 2)(\nu n_1 + p_2 - 2) $ is the dimension of the matrix $ A_{\bf n,p}^{\spQQ_{\balpha}} $.
\\ The use of \cite[Lemmas 6 and 8-11]{galerkin_gbsp}, along with Lemmas \ref{lem:poly01} and \ref{specdist_more}, and the fact that the matrices $ K_{n_i,p_i}^{\spQ_{\alpha_i}} $, $ H_{n_i,p_i}^{\spQ_{\alpha_i}} $, $ M_{n_i,p_i}^{\spQ_{\alpha_i}} $, $ B_{n_i,p_i}^{\spQ_{\alpha_i}} $, $ C_{n_i,p_i}^{\spQ_{\alpha_i}} $, $ i=1,2 $, are normal, yield
{\small
\begin{align*}
         \Vert R_{\bf n,p}^{\spQQ_{\balpha}} \Vert
& \leq   \frac{1}{\nu} \Vert \widehat{R}_{\bf n,p}^{\spQQ_{\balpha}} \Vert
       + \nu \Vert \widetilde{R}_{\bf n,p}^{\spQQ_{\balpha}} \Vert \\
& =      \frac{1}{\nu} \Vert \widehat{K}_{\bf n,p}^{\spQQ_{\balpha}} - \widehat{B}_{\bf n,p}^{\spQQ_{\balpha}} \Vert
       + \nu \Vert \widetilde{K}_{\bf n,p}^{\spQQ_{\balpha}} - \widetilde{B}_{\bf n,p}^{\spQQ_{\balpha}} \Vert \\
& \leq   \frac{1}{\nu} \Vert M_{n_2,p_2}^{\spQ_{\alpha_2}} \Vert \mbox{ } \Vert K_{n_1,p_1}^{\spQ_{\alpha_1}} \Vert
       + \frac{1}{\nu} \Vert C_{n_2,p_2}^{\spQ_{\alpha_2}} \Vert \mbox{ } \Vert B_{n_1,p_1}^{\spQ_{\alpha_1}} \Vert
       +          \nu  \Vert K_{n_2,p_2}^{\spQ_{\alpha_2}} \Vert \mbox{ } \Vert M_{n_1,p_1}^{\spQ_{\alpha_1}} \Vert
       +          \nu  \Vert B_{n_2,p_2}^{\spQ_{\alpha_2}} \Vert \mbox{ } \Vert C_{n_1,p_1}^{\spQ_{\alpha_1}} \Vert \\
& \leq   \frac{1}{\nu} \Vert M_{n_2,p_2}^{\spQ_{\alpha_2}} \Vert_{\infty} 
         \mbox{ } \Vert K_{n_1,p_1}^{\spQ_{\alpha_1}} \Vert_{\infty}
       + \frac{r}{\nu} M_{f_{p_1}}
       +          \nu  \Vert K_{n_2,p_2}^{\spQ_{\alpha_2}} \Vert_{\infty} 
         \mbox{ } \Vert M_{n_1,p_1}^{\spQ_{\alpha_1}} \Vert_{\infty}
       +       r  \nu  M_{f_{p_2}}.
\end{align*} }
\cite[Lemma 7]{galerkin_gbsp} implies the existence of a constant $Q_{p_1,p_2}$, independent of $n_1$, for which
\begin{equation} \label{eq:poly95}
\Vert R_{\bf n,p}^{\spQQ_{\balpha}} \Vert \leq Q_{p_1,p_2}.
\end{equation}
As a result, we can state a theorem which could be seen as the two-dimensional version of \cite[Theorem 5]{galerkin_gbsp}.
\begin{theorem} \label{th:poly18}
The sequence of matrices $ \{ A_{\bf n,p}^{\spQQ_{\balpha}} \} $ is distributed, in the sense of the eigenvalues, like the function $ g_{p_1,p_2} : [-\pi,\pi]^2 \rightarrow \mathbb{R} $
\begin{equation} \label{eq:poly96}
g_{p_1,p_2} := \frac{1}{\nu} h_{p_2} \otimes f_{p_1} + \nu f_{p_2} \otimes h_{p_1}.
\end{equation}
\end{theorem}
\begin{proof}
The proof is analogous to \cite[Theorem 18]{GMPSS14}, since matrices and functions have the same properties, and the same relations hold between them. 
\end{proof}
In the non-nested case, Theorem \ref{th:poly18} can be rendered as:
%
%
\begin{theorem} \label{th:poly18-nonnested}
The sequence of matrices $ \{ A_{\bf n,p}^{\spQQ_{\nn \balpha}} \} $ is distributed, in the sense of the eigenvalues, like the function $ g_{p_1,p_2}^{\spQ_{\alpha}} : [-\pi,\pi]^2 \rightarrow \mathbb{R} $
\begin{equation} \label{eq:poly96-nonnested}
g_{p_1,p_2}^{\spQ_{\alpha}} := \frac{1}{\nu} h_{p_2}^{\spQ_{\alpha}} \otimes f_{p_1}^{\spQ_{\alpha}} 
                             + \nu f_{p_2}^{\spQ_{\alpha}} \otimes h_{p_1}^{\spQ_{\alpha}}.
\end{equation}
\end{theorem}

\end{document}